\documentclass{amsart}

\usepackage{paralist}
\usepackage{amsmath}
\usepackage{amsfonts}
\usepackage{anysize}

\usepackage{bm}
\usepackage{bbm}
\usepackage{amsthm}
\usepackage{latexsym}
\usepackage{amssymb}
\usepackage{graphicx}

\usepackage{overpic}

\newcommand\note[1]{\marginpar{\small #1}} 

\newcommand\exampleEnd{\Diamond}

\newcommand\R{\mathbb{R}}               
\newcommand\Z{\mathbb{Z}}               
\newcommand\N{\mathbb{N}}               
\newcommand\Zk{\mathbb{Z}_k}            

\renewcommand\l{\ell}                   
\newcommand\Zl{\mathbb{Z}_{\l}}         

\newcommand\mTen{\theta}               
\newcommand\mFlow{\phi}                 
\newcommand\rmFlow{\#\rmFlowSet}            
\newcommand\iFlow{\overline{\phi}}      

\newcommand{\rmFlowSet}{\mathsf{F}} 

\newcommand\fb{\overline{f}}            

\newcommand\s{\sigma}                   

\newcommand\Tc{T^c}                     
\newcommand\oA{C}                       
\newcommand\A{A}                        
\newcommand\Grp{\mathcal{A}}            

\renewcommand\H{\mathcal{H}}              
\newcommand\B{\mathcal{B}}              

\renewcommand\P[1]{P_G(#1)}
\newcommand\Po[1]{P^\circ_G(#1)}
\newcommand\Pb{\P{b}}                  
\newcommand\Pbo{\Po{b}}           

\newcommand\bull{{\scriptscriptstyle\bullet}}

\newcommand\0{0}                        

\newcommand\rar{\rightarrow}

\newcommand\tcf{\#{\tcs}}
\newcommand\tcs{\mathsf{T}}
\newcommand\sprod[2]{\langle #1 , #2 \rangle}

\newcommand\defn[1]{\emph{#1}}
\newcommand\ehr{\mathsf{Ehr}}
\newcommand\supp{\mathsf{supp}}
\newcommand\relint{\mathsf{relint}}
\renewcommand\int{\mathsf{int}}
\newcommand\vol{\mathsf{vol}}
\newcommand\rk{\mathsf{rank}\,}
\renewcommand\dim{\mathsf{dim}\,}

\newcommand\1{\mathbbm{1}}

\newtheorem{thm}{Theorem}[section]
\newtheorem{cor}[thm]{Corollary}
\newtheorem{lem}[thm]{Lemma}
\newtheorem{prop}[thm]{Proposition}

\theoremstyle{definition}
\newtheorem{dfn}[thm]{Definition}

\newtheorem*{Ex1}{Example 1}
\newtheorem*{Ex2}{Example 2}

\title{Ehrhart theory, Modular flow reciprocity, and the Tutte polynomial}

\author{Felix Breuer}
\address{Felix Breuer, Freie Universit\"at Berlin, Institut f\"ur Mathematik,
Arnimallee 3, 14195 Berlin, Germany}
\email{felix.breuer@fu-berlin.de}

\author{Raman Sanyal}
\address{Raman Sanyal, Department of Mathematics, UC Berkeley, 970 Evans Hall,
Berkeley, CA 94720}
\email{sanyal@math.berkeley.edu}

\thanks{Felix Breuer was supported by the Deutsche Forschungsgemeinschaft
within the research training group `Methods for Discrete Structures' (GRK
1408). Raman Sanyal was supported by the Konrad-Zuse-Zentrum f\"ur
Informationstechnik Berlin and by a Miller Research Fellowship.}

\keywords{modular flow polynomial, reciprocity, Ehrhart theory, totally
cyclic orientations, Tutte polynomial}

\date{\today}

\begin{document}

\begin{abstract}
    Given an oriented graph $G$, the modular flow polynomial $\mFlow_G(m)$
    counts the number of nowhere-zero $\Z_k$-flows of $G$. We give a
    description of the modular flow polynomial in terms of (open) Ehrhart
    polynomials of lattice polytopes. Using Ehrhart-Macdonald reciprocity we
    give a combinatorial interpretation for the values of
    $\mFlow_G$ at negative arguments which answers a question of Beck and
    Zaslavsky (2006).  Our construction extends to $\Zl$-tensions and we
    recover Stanley's reciprocity theorem for the chromatic polynomial.
    Combining the combinatorial reciprocity statements for flows and tensions,
    we give an enumerative interpretation for positive evaluations of the
    Tutte polynomial $t_G(x,y)$ of $G$.
\end{abstract}

\maketitle

\section{Introduction}

The chromatic polynomial of a graph is probably the most famous graph
polynomial. In 1973 Stanley \cite{stanley73} gave an ``unorthodox''
interpretation of graph colorings in terms of acyclic orientations and
compatible maps. The benefit of this interpretation is a natural,
combinatorial interpretation of (suitably normalized) evaluations of the
chromatic polynomial at a negative argument. In some sense this was one of the
first \emph{combinatorial reciprocity theorems} \cite{stanley74}. In 2006,
Beck and Zaslavsky \cite{BZ06} gave a different perspective on this result by
casting it into the realms of geometry. They identified graph colorings as
lattice points ``inside'' a polytope but ``outside'' a hyperplane arrangement
--- an object answering to the name of \emph{inside-out polytope}. Thus, the
chromatic polynomial can be understood as a sum of Ehrhart functions and a
suitably generalized Ehrhart-Macdonald reciprocity yields the combinatorial
interpretation. We explain more of the details in the sections to come.

An equally important polynomial invariant of a graph is given by the 
\emph{modular flow polynomial}. Let $G = (V,E)$ be an \emph{oriented graph}
and let $\Grp$ be an abelian group. An \emph{$\Grp$-flow} is an assignment $f :
E \rightarrow \Grp$ such that at every vertex we have a conservation of flow,
i.e.~
\[
    \sum_{uv \in E} f_{uv} - \sum_{vu \in E} f_{vu} = 0
\]
for every $v \in V$. The \emph{support} of the flow $f$ is $\supp(f) = \{ e
\in E : f_e \not= 0 \}$ and $f$ is called \emph{nowhere-zero} if $\supp(f) =
E$.  Tutte \cite{tutte47} was the first to consider nowhere-zero flows for a
fixed group $\Grp$.  He proved that the number of nowhere-zero $\Grp$-flows
only depends on the order of the group and that $\mFlow_G(k)$, the number of
nowhere-zero $\Zk$-flows, is a polynomial in $k$.  Clearly, this is only
meaningful for finite groups, but in the case of $\Z$-flows a natural
concept is that of a \emph{$k$-flow} which is a $\Z$-flow with values strictly
smaller than $k$ in absolute values. Tutte \cite{tutte54} proved that there is
a nowhere-zero $\Z_k$-flow if and only if there is a nowhere-zero $k$-flow.
However, the number of nowhere-zero $k$-flows and $\Z_k$-flows differ in
general.

In 2002 Kochol \cite{kochol02} proved that $\iFlow_G(k)$, the number of
nowhere-zero $k$-flows, is also a polynomial and in \cite{BZ06-2} Beck and
Zaslavsky showed that this is yet another incarnation of Ehrhart theory of
inside-out polytopes. Moreover, this approach yields a reciprocity statement
that parallels that for the chromatic polynomial: $(-1)^{\xi(G)}\iFlow_G(-k)$
counts pairs of $k$-flows and compatible totally cyclic orientations. This
raised the question for a combinatorial reciprocity theorem of the modular
flow polynomial (cf.~\cite[Problem 3.2]{BZ06-2}).

As an answer to this question, the first result of this paper gives an
interpretation of $(-1)^{\xi(G)}\mFlow_G(-k)$ as naturally counting pairs of
$\Zk$-flows and totally cyclic reorientations on certain subgraphs. We give
the precise statement in Section~\ref{sec:Flows}. A little surprisingly, our
proof is somewhat simpler than the one for $k$-flows in \cite{BZ06-2}. For
starters, we do not need the theory of inside-out polytopes \emph{per se}; in
Section \ref{sec:FlowsInsideOut} we relate our proof to inside-out polytopes
which sheds ``geometric light'' on some well-known properties of flow
polynomials. In Section~\ref{sec:Tensions}, we discuss colorings and their
relations to $\Z_k$-tensions. We sketch how analogous arguments yield a
reciprocity statement for $\Z_k$-tensions that corresponds to Stanley's
reciprocity for
colorings \cite{stanley73}.  In Section~\ref{sec:TutteInterpretation} we make
use of the reciprocity statements to prove a enumerative interpretation for
arbitrary evaluations of Tutte polynomials of graphs at positive arguments,
which is implicit in the work of Reiner~\cite{reiner99}.  In the appendices we
give traditional, that is deletion-contraction, proofs for the main results of
Section~\ref{sec:Flows} and \ref{sec:TutteInterpretation}.

{\bf Acknowledgments.} We would like to thank Matthias Beck for valuable
conversations and comments on an earlier version of this paper.

\section{Modular flow reciprocity}
\label{sec:Flows}

Let $G = (V,E)$ be an \emph{oriented graph}, that is, an unoriented graph
equipped with an orientation of its edges.  We allow, even encourage, $G$ to
have multiple edges and loops. For an $S \subseteq E$ we denote by
$G_{\setminus S}$, $G_{/ S},$ and $G[S]$ the result of \emph{deleting},
\emph{contracting}, and \emph{restricting} to $S$, respectively.  Moreover, we
denote by $_SG$ the \emph{reorientation} of $G$ along $S$, i.e.~the graph
obtained by reversing the orientation of the edges in $S$. We denote by $c(G)$
the number of (weakly) connected components and we call $e \in E$ a
\emph{coloop} or \emph{bridge} if $c(G_{\setminus e}) = c(G) + 1$.  Finally,
we denote by $\xi(G) := |E| - |V| + c(G)$ the \emph{cyclotomic number} of $G$.

Let us give a precise definition for the main character.
\begin{dfn} For an oriented graph $G = (V,E)$,  the \emph{modular flow
    polynomial} $\mFlow_G$ of $G$ is the function
    \[
        \mFlow_G(k) = \# \left\{ f : E \rightarrow \Zk : f \text{ nowhere-zero
        $\Zk$-flow} \right\}.
    \]
\end{dfn}

The name was justified by Tutte \cite{tutte54} who showed that $\mFlow_G$ is
indeed a polynomial of degree $\xi(G)$.  In particular, $\mFlow_G$ can be
extended to negative arguments. In order to state our main result of this
section we need the notion of a totally cyclic orientation. An oriented graph
is called \emph{totally cyclic} if every edge is contained in a directed cycle
and $\s \subseteq E$ is a \emph{totally cyclic reorientation} if $_\s G$ is
totally cyclic.

\begin{thm}[Modular flow reciprocity]\label{thm:MainFlow}
    Let $G = (V,E)$ be an oriented graph and let $k$ be a positive integer.
    Then $(-1)^{\xi(G)}\mFlow_G(-k)$ counts pairs $(f,\s)$ where $f$ is a
    $\Zk$-flow and $\s \subseteq E \setminus \supp(f)$ is a totally cyclic
    reorientation for $G_{/ \supp(f)}$.
\end{thm}

Let us remark, that our result differs from the reciprocity theorem for
$k$-flows (cf.~\cite[Thm~3.1b]{BZ06-2}) inasmuch, that the flow and the
reorientation are not subject to a compatibility constraint. Rather, we reorient
at most those edges $e$ with $f(e)=0$ in the first place and contract all other
edges.

Let us illustrate the result with two examples that will accompany us
throughout.
\begin{Ex1}
    Consider the following graph $G_1$ with two vertices and three parallel
    and identically oriented edges $e_1,e_2,e_3$.
    \begin{center}
        \begin{overpic}[width=3cm]{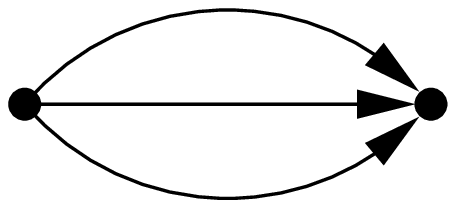}
            \put(45,8){\scriptsize$e_1$}
            \put(45,28){\scriptsize$e_2$}
            \put(45,48){\scriptsize$e_3$}
        \end{overpic}
    \end{center}
    The flow conservation for a flow $f$ is given by $f(e_1) + f(e_2) + f(e_3)
    = 0$ and this readily yields the number of nowhere-zero $\Zk$-flows as
    $\mFlow_{G_1}(k) = (k-1)(k-2)$.  The cyclotomic number of $G_1$ is
    $\xi(G_1) = 2$ and hence $(-1)^2\mFlow_{G_1}(-k) = (k+1)(k+2)$.  Now let
    us count the number of pairs $(f,\s)$ stated in Theorem~\ref{thm:MainFlow}
    according to $n_f = |\supp(f)|$. For $n_f = 0$, $f$ is the unique
    zero-flow and the number of totally cyclic orientations is $6$. The case
    $n_f = 1$ does not show up and for $n_f = 2$ there are exactly
    $\tbinom{3}{2}$ choices of non-zero edges and $(k-1)$ flows each time. The
    contraction in each case yields a loop which has two totally cyclic
    reorientations. Together with $n_f = 3$, in which case we count the number
    of nowhere-zero $\Zk$-flows, we get 
    $
        (k-1)(k-2) + 6 (k-1) + 6 = (k+1)(k+2).
    $ 
    \hfill $\exampleEnd$
\end{Ex1}

\begin{Ex2}
    Our second example is the multigraph $G_2$:
    \begin{center}
        \begin{overpic}[width=3cm]{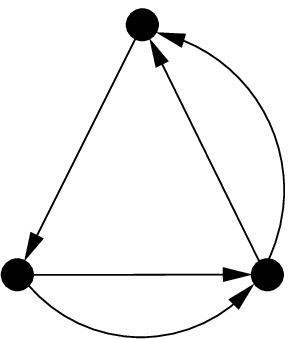}
            \put(17,57){$a$}
            \put(62,57){$b$}
            \put(85,57){$b^\prime$}
            \put(42,22){$c$}
            \put(42,3){$c^\prime$}
        \end{overpic}
    \end{center}
    In order to compute the flow polynomial, consider the case of a $3$-cycle
    $a,b,c$ without parallel edges. It is clear that the flow is determined by
    the (non-zero) value on the edge $a$ and the flow polynomial is thus
    $k-1$. Now every nowhere-zero flow on the $3$-cycle yields $(k-2)^2$
    nowhere-zero flows on $G_2$ since flow on $b$ and $c$ can be ``rerouted''
    through $b^\prime$ and $c^\prime$ as long as all remain nonzero. Hence,
    $\mFlow_{G_2}(k) = (k-1)(k-2)^2$ and $(-1)^{\xi(G_2)}\mFlow_{G_2}(-k) =
    (k+1)(k+2)^2$ with $\xi(G_2) = 3$.  The argument extends to counting the
    pairs $(f,\s)$ combinatorially, by lifting the pairs from the $3$-cycle.
    \hfill $\exampleEnd$
\end{Ex2}

We will now set the stage for the proof of Theorem~\ref{thm:MainFlow} which
will mainly consist of casting the statement of Theorem~\ref{thm:MainFlow}
into a discrete geometric statement involving lattice polytopes. As a first
step we will identify $\Zk$ with a set of coset representatives given by the
integers $0,1,\dots,k-1$. With this identification the flow conservation at a
vertex $v \in V$ can be rephrased as
\begin{equation*} 
\begin{aligned}
    \sum_{uv \in E} f_{uv} &\;-\; \sum_{vu \in E} f_{vu} \ = \ 0 & \text{ over
    $\Zk$}\\
    \Leftrightarrow \sum_{uv \in E} f_{uv} &\;-\; \sum_{vu \in E} f_{vu} \
    \equiv \ 0 & \text{ mod $k$ }\\
    \Leftrightarrow \sum_{uv \in E} f_{uv} &\;-\; \sum_{vu \in E} f_{vu} \
    = \ k \cdot b_v & \text{ for some  } b_v \in \Z.\\
\end{aligned}
\end{equation*}

Letting $A = A_G \in \{0,\pm1\}^{V \times E}$ be the \emph{incidence matrix}
of $G$, the last equivalence yields the following polyhedral reformulation: A
point $f \in \Z^E$ represents a nowhere-zero $\Zk$-flow if there is a $b \in
\Z^V$ such that $f$ is contained in $(k\cdot P^\circ_G(b)) \cap \Z^E$ where
\[
    P^\circ_G(b) := \{ p \in \R^E : Ap = b, 0 < p_e < 1 \text{
    for all } e \in E \}.
\]
Note that for such a $b$, $P^\circ_G(b)$ is a relatively open polytope of
dimension $\dim \Pbo = \rk A_G = \xi(G)$. Denote by $\B_G = \{ b \in \Z^V :
\Pbo \not= \emptyset \}$ the collection of all \emph{feasible} $b$'s. The set
$\B_G$ is clearly finite (since the cube is compact) and for distinct $b,
b^\prime \in \B_G$, the relatively open polytopes $\Pbo$ and $\Po{b^\prime}$ are necessarily
disjoint. The incidence matrix $\A_G$ of an oriented graph is \emph{totally
unimodular} (cf.,~for example, \cite[Sec.~19.3,Ex 2]{schr86}). This remains true
if we add rows that contain just a single $1$ to encode constraints like
$0\leq p_e\leq 1$. Standard methods (cf.  \cite[Thm.~19.1]{schr86}) then imply
that the closure $\Pb = \overline{\Pbo}$ is a vertex induced subpolytope of
the $|E|$-dimensional standard cube and, in particular, a \emph{lattice}
polytope.

\begin{Ex1}[continued]
The edge space of $G_1$ is three dimensional. The incidence matrix of $G_1$ is 
\[
    A = A_{G_1} =
    \left(\begin{array}{rrr} -1 & -1 & -1 \\ 1 & 1 & 1 \end{array}\right)
\] 
It follows that $P^\circ_{G_1}(b) = \{ x \in \R^3 : 0 < x_1,x_2,x_3 < 1,
x_1+x_2+x_3 = -b_1, x_1 + x_2 + x_3 = b_2 \}$ is non-empty iff $b^\prime = (-1,1)$ or
$b^{\prime\prime} = (-2,2)$. The following figure shows the two polytopes as
slices of the cube. The points correspond to the $6$ nowhere-zero
$\Z_4$-flows.
\begin{center}
    \hfill\includegraphics[width=7cm]{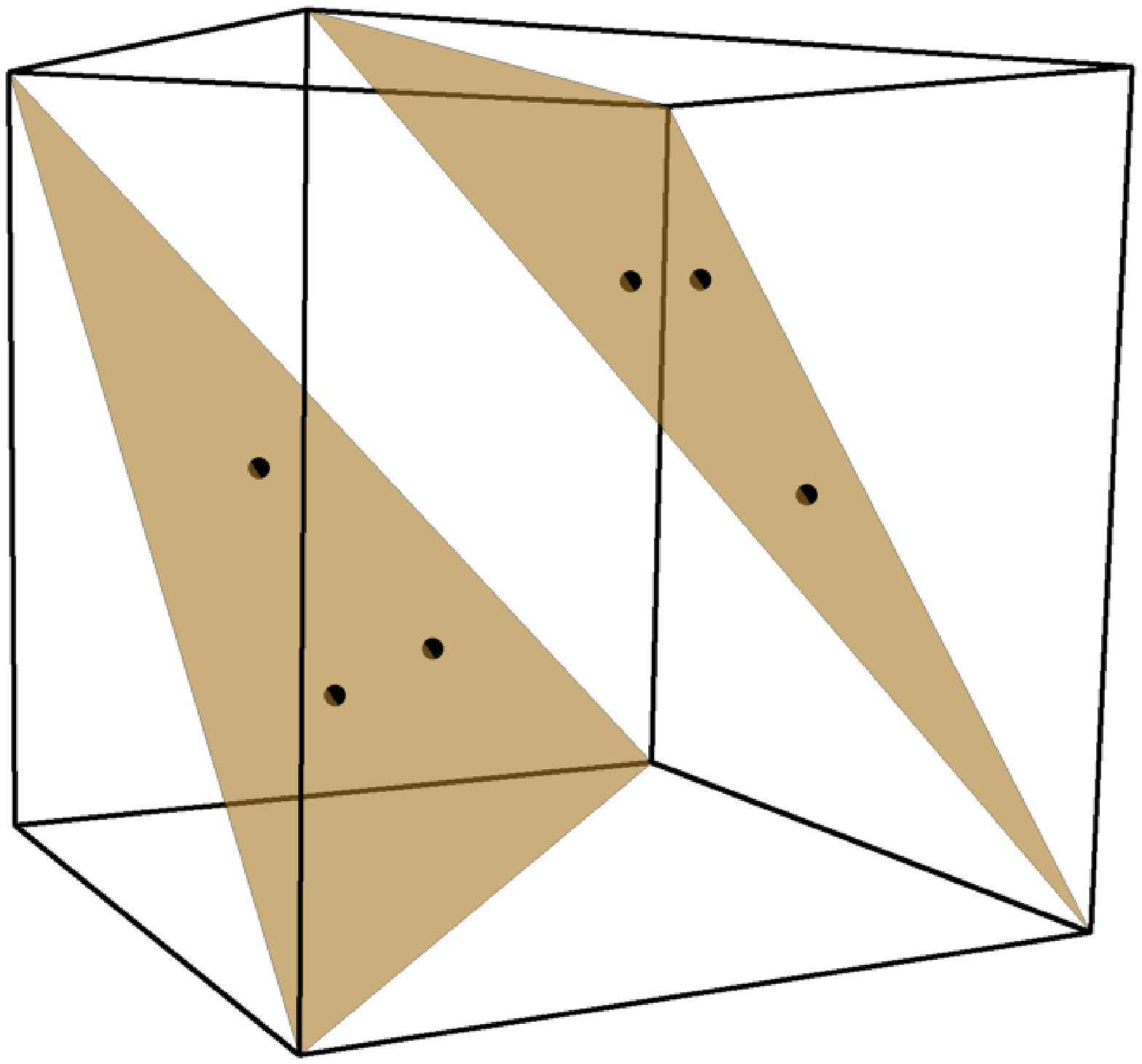}
     \hfill $\exampleEnd$
\end{center}
\end{Ex1}

For a polytope $P \subset \R^d$ the function $\ehr(P;k) := (k \cdot P) \cap
\Z^d$ is called the \emph{Ehrhart function} of $P$. Ehrhart \cite{ehr77}
showed that $\ehr(P;k)$ is a polynomial of degree $\dim P$ in case $P$ is a
lattice polytope. 

\begin{prop}\label{prop:EhrFlow}
    Let $G = (V,E)$ be an oriented graph. Then
    \[
        \mFlow_G(k) = \sum_{b \in \B_G} \ehr(\Pbo;k)
    \]
    is a sum of Ehrhart polynomials.\qed
\end{prop}

Probably the most appealing feature of Ehrhart theory is that Ehrhart
polynomials adhere to a beautiful geometric reciprocity (for details
see e.g.~\cite[Sect.~4]{BR07}).

\begin{thm}[Ehrhart--Macdonald reciprocity] \label{thm:EM}
    Let $P$ be a rational polytope and denote by $P^\circ$ the (relative) interior of
    $P$. Then
    \[
        \ehr(P^\circ;k) = (-1)^{\dim P} \ehr(P;-k).
    \]
\end{thm}

In light of Proposition~\ref{prop:EhrFlow} together with  Ehrhart-Macdonald
reciprocity it is sufficient to give a combinatorial meaning to the lattice
points in the boundary of $k \cdot \Pb$. Fix a $b \in \B_G$ and let $F \subset
\Pb$ be a proper face. As $\Pb$ is a section of the $|E|$-cube, there is a
partition $\s_- \cup \s_0 \cup \s_+ = E$ into disjoint parts such that the
relative interior of $F$ is given by all the points $p \in \Pb$ such that
\[
    \begin{array}{cccccl}
        0 &=& p_e & &   & \text{ for } e \in \s_-,\\
        0 &<& p_e &<& 1 & \text{ for } e \in \s_0, \text{ and }\\
          & & p_e &=& 1 & \text{ for } e \in \s_+.\\
    \end{array}
\]
A lattice point $f$ in the relative interior of $k \cdot F$ represents a
$\Zk$-flow but, since $0 \equiv k\; \mathsf{mod}\, k$, this representation is
not unique.  However, if $f^\prime \in (k \cdot \P{b^\prime}) \cap \Z^E$ for
some $b^\prime \in \B_G$ yields the same $\Zk$-flow modulo $k$ but is
different from $f$ otherwise, then $f^\prime \in \relint(k \cdot F^\prime)$
for a proper face $F^\prime \subset P_G(b^\prime)$ and $F \not= F^\prime$.
Thus, the idea is to keep track of the \emph{origin} of $f$. This leads to
reorientations on the contraction.

To this end, let $z_F$ and $z_P$ be points in the relative interiors of $F$
and $\Pb$ respectively and consider $z := z_P - z_F$. Then $z \in \ker A_G$
and we can predict the sign of $z_e$ for $e \in \s_- \cup \s_+$.  The next
lemma relates the kernel of $A_G$ to totally cyclic reorientations of $G$.

\begin{lem}[{\cite[Lem.~8.1]{GZ83}}] \label{lem:TotCyc}
    Let $G = (V,E)$ be a graph and $A = A_G$ its incidence matrix. The
    connected components of $ \ker A \setminus \{ p \in \R^E : p_e = 0 \text{
    for some } e \in E \}$ are in bijection with the totally cyclic
    reorientations of $G$. The totally cyclic reorientation $\s$ associated to
    a connected component is $\s = \{ e \in E : p_e < 0 \}$ for an arbitrary
    point $p$ in that component.
\end{lem}

The restriction $\tilde{z}$ of $z$ to $\s_+ \cup \s_-$ is an element of $\ker
A_{G / \s_0}$ and Lemma~\ref{lem:TotCyc} then asserts that $\s_+$ is a totally
cyclic reorientation for $G_{/\s_0}$ and uniquely identifies the face $F$ for
which $z_F \in \relint\,F$.

\begin{proof}[Proof of Theorem~\ref{thm:MainFlow}]
    Let $f \in k \cdot [0,1]^E \cap \Z^E$ be a lattice point in the $k$-th
    dilate of the $|E|$-dimensional standard cube. Let $\s(f) = \{ e \in E :
    f_e = k \}$ and denote by $\fb$ the point $f$ modulo $k$ componentwise.
    
    By our discussion and, in particular, Ehrhart-Macdonald reciprocity we
    have that
    \[
    (-1)^{\xi(G)}\mFlow_G(-k) = \left| \bigcup_{b \in \B_G} (k \cdot \Pb) \cap
    \Z^E \right|
    \]
    where the union on the right-hand side is over disjoint sets. The
    theorem follows by proving that $f \mapsto (\fb,\s(f))$ is a bijection
    between points in the right-hand side and pairs of $\Zk$-flows and totally
    cyclic orientations on the contraction of the support. However, by our
    previous reasoning it is clear that this is a well-defined map and we are
    left with showing that there exists an inverse mapping.

    Let $(\fb,\s)$ be a pair with $\fb : E \rightarrow \Zk$ a $\Zk$-flow and
    $\s \subseteq E \setminus \supp(\fb)$ a totally cyclic reorientation of
    $G_{/ \supp(\fb)}$. Let $f^\prime \in k \cdot [0,1]^E \cap \Z^E$ be the
    unique point with $f^\prime_e = k$ iff $e \in \s$ and $\overline{f^\prime}
    = \fb$. The point $f^\prime$ is in the boundary of $\Pb$ for $b =
    Af^\prime$ and we are done if we can show that $b \in \B_G$. Now, by
    Lemma~\ref{lem:TotCyc}, we can pick a vector $z \in \R^E$ with $Az = \0$,
    $z_e < 0$ for $e \in \sigma$ and $z_e > 0$ for $e \not\in \supp(\fb) \cup
    \s$.  Thus, for $\varepsilon > 0$ sufficiently small, $f^\prime +
    \varepsilon z$ is a point of $\Pb$ in the interior of the cube and this
    concludes the argument.  
\end{proof}

As an immediate Corollary we get the following known enumerative result.

\begin{cor}[{\cite[Cor.~1.3]{stanley73}}]
    Let $G = (V,E)$ be an oriented graph. Then $(-1)^{\xi(G)}\mFlow_G(-1)$ is
    the number of totally cyclic reorientations of $G$.
\end{cor}

In particular, every totally cyclic reorientation belongs to exactly one
$\Pb$. It is worthwhile interpreting this partition of totally cyclic
reorientations as an equivalence relation.  The following proposition phrases
the equivalence in combinatorial terms.  For a set $\sigma \subseteq E$ we
denote by $e_\sigma \in \{0,1\}^E$ the characteristic vector of $\sigma$.

\begin{prop}
    Let $G = (V,E)$ be an oriented graph and $\sigma, \sigma^\prime \subseteq
    E$ two totally cyclic reorientations.  The points $e_\sigma,
    e_{\sigma^\prime} \in \{0,1\}^E$ are both vertices of $\Pb$ for some $b
    \in\B_G$ if and only if $_\sigma G$ can be obtained from $_{\sigma^\prime}
    G$ by the reversal of directed cycles.
\end{prop}
\begin{proof}
    The points $e_\sigma, e_{\sigma^\prime} \in \{0,1\}^E$ are both contained
    in a common $\Pb$ iff $z := e_\sigma - e_{\sigma^\prime}$ is an element of
    $\ker A_G$. By \cite[Lem.~8.5]{GZ83}, $z$ is a linear combination with
    non-negative coefficients of orientations of cycles of $G$.
\end{proof}

For an oriented graph $G = (V,E)$, we denote by $I_G \in \Z^V$ the
\emph{in-degree sequence}, that is the number $(I_G)_v$ of incoming edges for
every vertex $v \in V$. Similar, we define the \emph{out-degree sequence} $O_G
\in \Z^V$ of~$G$.

\begin{thm}
    Let $G = (V,E)$ be an oriented graph. Then $\B_G$ is in bijection with 
    \[
        \{ I_{_\sigma G} : \sigma \subseteq E \text{ totally cyclic
        reorientation} \}.
    \]
\end{thm}

\begin{proof}
    Let $A_G$ be the incidence matrix of $G$. It is clear that we can recover
    $I_G$ from the knowledge of the (undirected) degree sequence $D = I_G +
    O_G$ and $I_G - O_G = A_G \1 =: b_0$. Now, for any reorientation $\sigma
    \subseteq E$, we have
    \[
        I_{_\sigma G} - O_{_\sigma G} = A_{_\sigma G} \1 = A_G (\1 - 2
        e_\sigma) = b_0 - 2A_G e_\sigma.
    \]
    Hence, the in-degree sequence of $_\sigma G$ is uniquely determined by
    $b_\sigma = A_G e_\sigma$.  Moreover, the reversal of a directed cycle in
    $G$ leaves the in- and out-degree sequences invariant and thus is
    invariant within $\Pb$. Using that $\B_G = \{ b = A_G e_\sigma : \sigma
    \subseteq E \text{ totally cyclic reorientation } \}$ finishes the proof.
\end{proof}

Dilating a lattice polytope by a factor $0$ yields a lattice point and it can
be shown that indeed the Ehrhart polynomial has constant term equal to $1$.
As $\mFlow_G$ is the sum of the (open) Ehrhart polynomials for $\Pbo$, we
recover the following result by Gioan~\cite{gioan07}. 

\begin{cor}[{\cite[Thm.~3.1]{gioan07}}]
    Let $G = (V,E)$ be an oriented graph and $t_G(x,y)$ its Tutte polynomial.
    Then 
    \[
    t_G(0,1) = (-1)^{\xi(G)} \mFlow_G(0) = \sum_{b \in \B_G} \ehr(\Pb;0) =
    |\B_G|
    \]
     is the number of in-degree sequences of
    totally cyclic reorientations of $G$.
\end{cor}

The evaluation $t_G(0,1)$ has several known interpretations as, for example, the number of spanning trees with zero external activity,
the Euler characteristic of the independence complex of the matroid $M_G$
associated to $G$, or the number of facets of the broken circuit complex of
the dual matroid $M_G^\perp$. We refer the reader to the survey article by
Brylawski and Oxley \cite{BO92} for further details.

\section{Modular flows inside-out}
\label{sec:FlowsInsideOut}

In this section we relate our previous construction to inside-out
polytopes. The benefit in doing so will be a simple geometric explanation for
the following fact.

\begin{cor}
    The modular flow polynomial of an oriented graph $G$ has degree $\xi(G)$
    and leading coefficient $1$.
\end{cor}

This, in turn, is a consequence of the fact that the number of modular
$\Zk$-flows is a \emph{generalized Tutte--Grothendieck invariant}
(cf.\ Section~\ref{sec:TutteInterpretation}) and, hence, obeys the following
deletion-contraction property.

\begin{prop}[{\cite[Prop.~6.3.4]{BO92}}] \label{prop:FlowRecursion}
    Let $G = (V,E)$ be an oriented graph and $e \in E$ an edge. If $e$ is
    neither a loop nor a coloop, then
    \[
    \mFlow_G(k) = \mFlow_{G_{/e}}(k) - \mFlow_{G_{\setminus e}}(k).
    \]
    Otherwise, $\mFlow_G(k) = (k-1)\mFlow_{G_{\setminus e}}(k)$ if $e$ is a loop
    and $\mFlow_{G}(k) = 0$ if $e$ is a coloop.
\end{prop}

While the degree of the polynomial is clear from our interpretation in terms
of Ehrhart polynomial, the fact that the leading coefficient is $1$ is not.

\begin{prop}[{\cite[Cor.~3.20]{BR07}}]
    Let $P$ be a $d$-dimensional lattice polytope then the leading coefficient
    of $\ehr(P;k)$ is the volume $\vol(P)$ of $P$.
\end{prop}

So, the best we can say so far is that
\[
    1 = \sum_{b \in \B_G} \vol(\Pb).
\]
However, the answer we are aiming at is that suitably arranging the polytopes
$\Pb$ yields a subdivision of a standard cube of dimension $\xi(G)$. Thus, the
total volume of the $\Pb$'s is that of the standard cube. 

In particular, the subdivision of the cube is induced by an arrangement of
hyperplanes and therefore directly leads to inside-out polytopes.  For our
needs a \emph{(lattice) inside-out polytope} is a pair $(P,\H)$ consisting of
a $d$-dimensional lattice polytope $P \subset \R^d$ and a hyperplane
arrangement $\H = \{H_1,\dots,H_m\}$ such that any flat that meets $P$ also
meets the interior of $P$.  The hyperplane arrangement is allowed to be
infinite as long as only finitely many hyperplanes meet $P$.  The \emph{open
Ehrhart function} of $(P,\H)$ is the function $\ehr(P^\circ,\H;k) = | k
\cdot(P^\circ \setminus \cup\H)  \cap \Z^d |$. In the case of a lattice
inside-out polytope, i.e. $\H$ subdivides $P$ into lattice polytopes, the open
(as well as the closed) Ehrhart function is a polynomial of degree $d$ and
leading coefficient $\vol(P)$. 

Let $G = (V,E)$ be an oriented graph and let $T \subseteq E$ be a spanning
forest, i.e.,~a spanning tree per component.  Let $\Tc = E \setminus T$ denote
the edges not in the forest. We will use $T$ to construct a cycle basis,
i.e.,~a basis for $\ker \A_G$, combinatorially. Let $\oA \in \{0,\pm1\}^{ \Tc
\times E}$ be the matrix with rows $\oA_{f\bull}$ for $f \in \Tc$ defined as
follows: The support of $\oA_{f\bull}$ is given by the edges in the unique
undirected cycle $K$ in $G[T \cup f]$. There is a unique reorientation of $K$
that makes it an oriented cycle and that fixes the orientation on $f$. The
signs keep track of which edges have to be reoriented.  It is known
(e.g.~\cite[Thm.~11.1]{Schri03}) that every flow is an integral linear
combination of the $C_{f\bull}$.  This implies that every nowhere-zero
$\Zk$-flow corresponds to a unique point $h \in \Z^{\Tc}$ with $0 < h_f < k$
for $f \in \Tc$ and $(h\oA)_e \not\equiv 0 \mod k$ for all $e \in T$. Denote
by $\oA_{\bull e}$ the columns of $\oA$ for $e \in T$, then this yields the
following interpretation in terms of inside-out polytopes.

\begin{prop}
    Let $P = [0,1]^{\Tc}$ be the standard cube in $\R^{\Tc}$ and let $\H$ be
    the arrangement of hyperplanes $H_{e,d_e} = \{ q \in \R^{\Tc} : 
    \oA_{\bull e}^T\,  q  = d_e \}$ for $e \in T$ and $d_e \in \Z$. Then
    \[
        \mFlow_G(k) = \ehr(P^\circ,\H;k).
    \]
\end{prop}

In some sense the above construction is dual to that presented in 
Section~\ref{sec:Flows}: whereas the linear conditions in the construction of 
Section~\ref{sec:Flows} forced the points to be $\Zk$-flows, the construction
here satisfies that automatically, and violating the linear conditions enforces
the nowhere-zero condition. This is an instance of oriented matroid duality
of the hyperplane arrangement $\H$.

\begin{Ex1}[continued]
    For the choice of the spanning tree $T = \{e_1\}$, the cycle basis is
    \begin{align*}
        &\hphantom{\Bigl(}\begin{array}{rrr} \scriptstyle e_1 & \scriptstyle
            e_2 & \scriptstyle e_3 \end{array}\\
        \oA = 
        \begin{array}{rr} 
            \scriptstyle e_2 \\ 
            \scriptstyle e_2 \\ 
        \end{array}
        &
        \Bigl(\!
        \begin{array}{rrr} -1 & 1 & 0 \\ -1 & 0 & 1 \end{array}\!\Bigr).
    \end{align*}
    The polytope $P$ is a $2$-cube and the hyperplane
    arrangement $\H$ is given by $H_{e_1,d} = \{ (q_1,q_2) \in \R^{\Tc} : q_1 +
    q_2 = -d \}$ for $d \in \Z$. The resulting inside-out polytope is the
    following.
    \begin{center}
        \includegraphics[width=4cm]{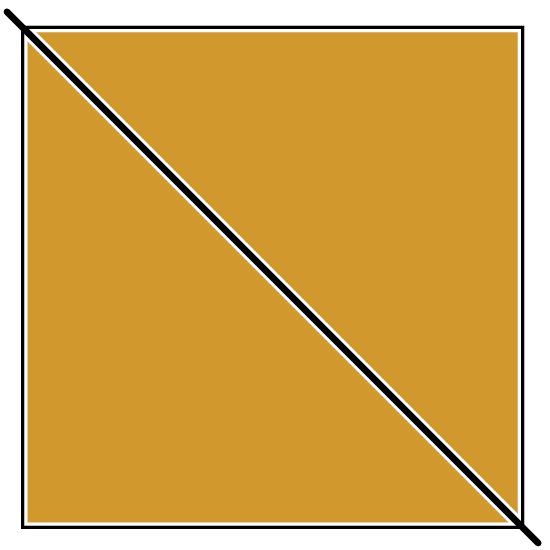}
    \end{center}
     \hfill $\exampleEnd$
\end{Ex1}

\begin{Ex2}[continued]
    For $G_2$ we pick the spanning tree $T = \{a,b\}$ and, hence, the cycle
    basis
    \begin{align*}
        &\hphantom{\Bigl(}\begin{array}{rrrrr} 
            \scriptstyle a & 
            \hspace{1em}\scriptstyle b & 
            \hspace{.3em}\scriptstyle b^\prime & 
            \scriptstyle c &
            \scriptstyle c^\prime 
        \end{array}\\
        \oA = 
        \begin{array}{l} 
            \scriptstyle b^\prime \\ 
            \scriptstyle c\\
            \scriptstyle c^\prime\\
        \end{array}\!\!
        &
        \Biggl(\!
        \begin{array}{rrrrr} 
            & -1  & 1 &   &   \\
          1 &  1&   & 1 &   \\
          1 &  1&   &   & 1 \\
        \end{array}\!\Biggr).
    \end{align*}
    So the two parallel classes of hyperplanes correspond to $H_{a,d_a} = \{
    (b^\prime,c,c^\prime) \in \R^{\Tc} : c + c^\prime = d_a \}$ and
    $H_{b,d_b} = \{ (b^\prime,c,c^\prime) \in \R^{\Tc} : -b^\prime + c +
    c^\prime = d_b \}$. The resulting inside-out polytope is:
    \begin{center}
        \includegraphics[width=8cm]{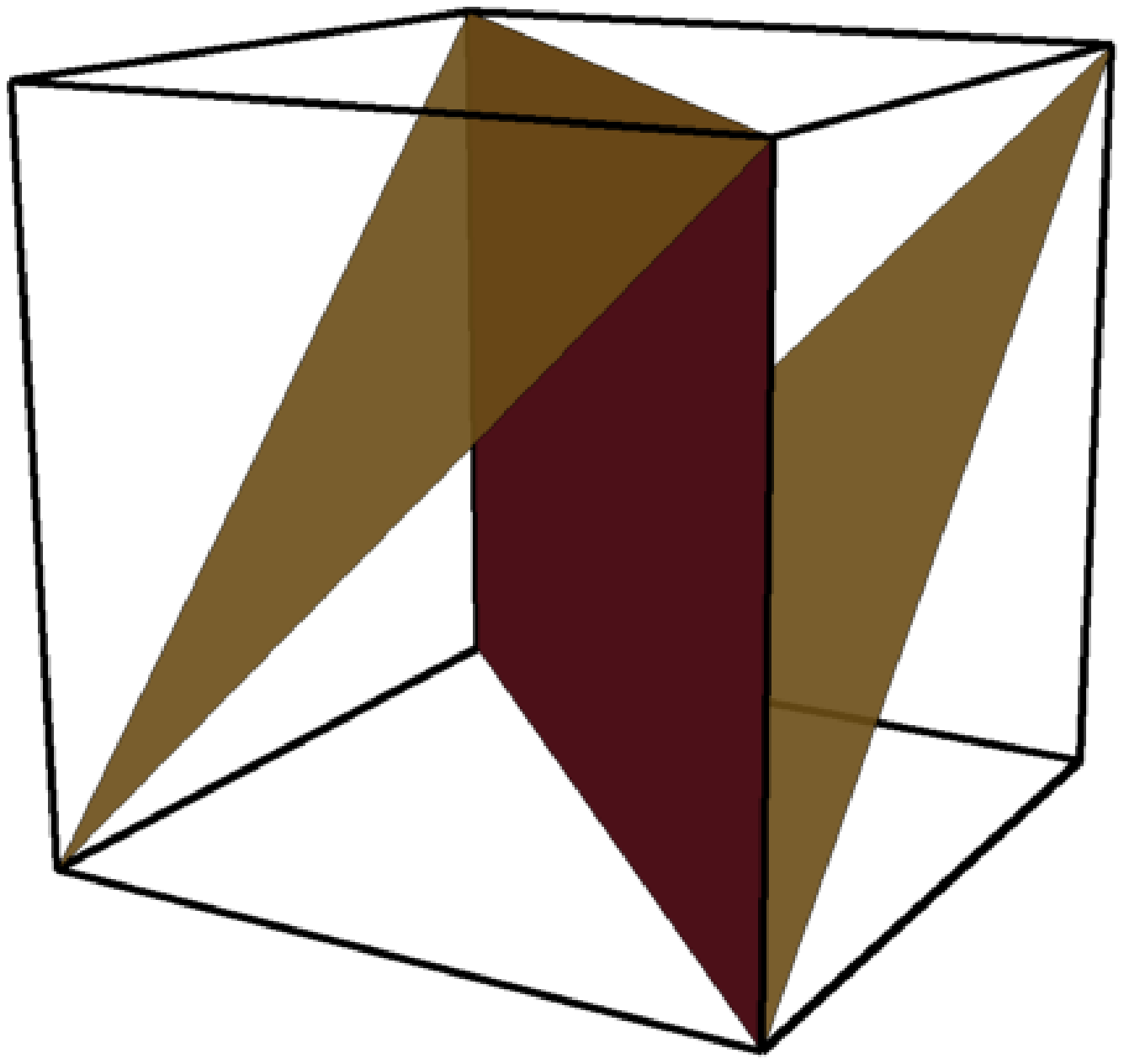}
    \end{center}
    The four chambers correspond to the feasible $b$'s.
     \hfill $\exampleEnd$
\end{Ex2}

So the leading coefficient of $\mFlow_G$ refers to the volume of a
$\xi(G)$-dimensional standard cube $P = [0,1]^{\Tc}$.  The chambers,
i.e.,~connected components of $P^\circ \setminus \cup\H$ are in bijection with
$\B_G$. Indeed, every chamber is isomorphic to a unique $\Pbo$ under the
restriction to the coordinates $\Tc$.  With some more work it is possible to
see the reciprocity in this picture and, via toric arrangements, this line of
thought has been pursued by Babson and Beck \cite{BB09}. One point worth
mentioning is that the deletion-contraction property of the flow polynomial
can be nicely observed in terms of the freedom of choice for a spanning
forest.  The contraction of an edge $e$ simply removes the family of
hyperplanes $H_{e,d_e}$ from the arrangement. The resulting over count can be
compensated for by subtracting the number of lattice points in the inside-out
polytope obtained by restricting to this family. However, choosing a spanning
forest not containing $e$ yields an inside-out polytope where the family
$H_{e,d_e}$ corresponds to two parallel facets of the cube and restricting
yields the inside-out polytope for $G_{\setminus e}$.

\section{Modular tensions and Stanley's Reciprocity Theorem}
\label{sec:Tensions}

In the introduction we mentioned Stanley's reciprocity theorem for the
chromatic polynomial $\chi_G$ of a graph $G=(V,E)$. It turns out that this
reciprocity is best observed in the related setting of modular tensions. In
this section we sketch the changes to our previous constructions in order to
accommodate modular tensions and we formulate a reciprocity theorem. 

An \defn{$\l$-coloring} is a map $c: V \rar \{0,\ldots,\l-1\}$. An
$\l$-coloring $c$ is called \defn{proper} if $c(u)\not=c(v)$ whenever $u$ and
$v$ are adjacent in $G$ and  the chromatic polynomial $\chi_G(\l)$ counts the
number of proper $\l$-colorings of $G$. A reorientation $\s$ of $G$ is
\defn{acyclic} if no edge of $G$ lies on a directed cycle. A coloring $c$ and
an acyclic reorientation $\s$ are \defn{compatible} if for every edge $uv$ of
$_\s G$ we have $c(u)\leq c(v)$.

\begin{thm}[Stanley \cite{stanley73}]\label{thm:StanleyReciprocity}
    For an oriented graph $G=(V,E)$ and $\l > 0$ an integer, 
    $(-1)^{|V|}\chi_G(-\l)$ is the number of tuples $(c,\s)$ where $c$ is an
    $\l$-coloring and $\s \subseteq E $ is a compatible acyclic reorientation.
\end{thm}

The idea that leads to the notion of tensions is the following. Let us suppose
for a moment that $G$ is connected. Then we can recover the coloring $c$ by
knowing the \emph{initial} color $c_0 = c(v_0)$ of some vertex $v_0$ and the
difference of the colors $t(uv) := c(u) - c(v)$ on each (oriented) edge $uv$.
Of course, we now make use of the fact that our set of colors is embedded in
a group.  Thus, we recover the color on a vertex $w$ by adding and
subtracting the tensions $t(e)$ of edges along a path from $v_0$ to $w$.
Note that the color of $c(w)$ is independent of the chosen path and we take
this as the defining property of tensions. 

For a cycle $C \subseteq E$ in the underlying undirected graph, we denote by
$C_- \subset C$ a collection of edges whose reorientation turns $C$ into a
directed cycle and we let $C_+ := C \setminus C_-$ be the remaining edges.
Let $\Grp$ be an abelian group. A map $t : E \rar \Grp$ is called an
\emph{$\Grp$-tension} if 
\[
    \langle C, t \rangle := \sum_{e \in C_+} t(e) - \sum_{e \in C_-} t(e) = 0
\]
for each undirected cycle $C \subseteq E$. Note that this is independent of
the choice of $C_-$.  An induced coloring $c$ is proper if and only if $t$ is
\emph{nowhere-zero}, that is $\supp(t) := \{ e \in E : t(e) \not= 0 \} = E$.
It can be shown that, as in the case of flows, the number of nowhere-zero
$\Grp$-tensions depends only on the order of $\Grp$ and we define
$\mTen_G(\l)$ to be the number of nowhere-zero $\Zl$-tensions of $G$.

Every nowhere-zero $\Zl$-tension $t$ yields $|\Zl|^{c(G)}$ different colorings
by choosing an initial color for every component and this proves the following
known fact that $\chi_G(\l)=\l^{c(G)} \mTen_G(\l)$, i.e.\ the tension
polynomial is a non-trivial factor of the chromatic polynomial.  Hence, we
arrive at the following equivalent reformulation of
Theorem~\ref{thm:StanleyReciprocity} which is already implicit in Stanley's
work.

\begin{thm}\label{thm:TensionReciprocity}
    Let $G=(V,E)$ be an oriented graph.
    Then $(-1)^{|V|-c(G)}\mTen_G(-\l)$ counts pairs $(t,\s)$ where $t$ is
    a $\Zl$-tension on $G$ and $\s$ is an acyclic reorientation of
    $G_{\backslash\supp(t)}$.
\end{thm}

\begin{proof}[Equivalence of Theorems~\ref{thm:StanleyReciprocity} and \ref{thm:TensionReciprocity}]
In light of the fact that $\chi_G(\l)=\l^{c(G)} \mTen_G(\l)$ it suffices to
argue that for a given $k$-coloring $c$ and a corresponding $\Zk$-tension $t$,
the acyclic reorientations of $G_{\setminus\supp(t)}$ are in bijection with
the acyclic reorientations of $G$ that are compatible with $c$. If $\s$ is an
acyclic reorientation of $G$ then clearly $\s\cap E\setminus \supp(t)$ is an
acyclic reorientation of $G_{\setminus\supp(t)}$. Conversely, let $\s'$ be an
acyclic reorientation  of $G_{\setminus\supp(t)}$. We have to show that there
is a unique extension of $\s'$ to an acyclic reorientation $\s$ of $G$ that is
compatible with $c$. However, the condition that $e=uv$ has to be oriented
from $u$ to $v$ whenever $c(u)<c(v)$ fixes the reorientation of all edges in
$\supp(t)$. Suppose the resulting reorientation $\s$ did contain a directed
cycle $C$. Then all vertices on $C$ have to have the same color with respect
to $c$, as following an edge can never decrease the color. But this means that
$C\subseteq E\setminus\supp(t)$ which is a contradiction to
$_{\s'}G_{\setminus\supp(t)}$ being acyclic.
\end{proof}

Note the similarity to the statement of Theorem~\ref{thm:MainFlow}. This
is not at all surprising from the (oriented) matroid point of view as
we have the following correspondences:
\begin{center}
\begin{tabular}{r@{$\;\;\leftrightarrow\;\;$}l}
    flow $f$ & tension $t$ \\
    $\xi(G)$ & $|V| - c(G)$\\
    totally cyclic & acyclic\\
    $G_{/\supp(f)}$ & $G_{\backslash \supp(t)}$ \\
\end{tabular}
\end{center}

Let us quickly remark on the conception of a ``tension''. If $\Grp$ is an
ordered group, such as $\Z$, then the elements in $\Grp$ can be thought of
heights and a tension measures the difference in altitude along an edge. In
particular, if we reorient $G$ such that $t(e) > 0$ for all $e$, then it is
clear that every nowhere-zero $\Z$-tension yields an acyclic reorientation. 

We refrain from giving a proof of Theorem~\ref{thm:TensionReciprocity} but
instead supply the necessary modifications to the proofs for flow reciprocity
in Sections~\ref{sec:Flows} and \ref{sec:FlowsInsideOut}.

For an arbitrary spanning tree $T$, let $C$ be the cycle basis constructed in
Section~\ref{sec:FlowsInsideOut}. It suffices to verify the defining
properties of tensions on such a cycle basis as every cycle is a
superposition of these elementary cycles. Thus, identifying $\Zl = \{ 0,\dots,
\l-1\}$, we have that $t : E \rar \Z$ represents a nowhere-zero $\Zl$-tension if
$0 < t(e) < \l$ for every $e \in E$ and there is a $d \in \Z^{\Tc}$ such that
$Ct = \l \cdot d$. Analogously to the flow case, this formulation furnishes a
collection of relatively open disjoint polytopes whose Ehrhart polynomials
yield the nowhere-zero tension polynomial. 

For the reciprocity the key lemma that yields the interpretation in terms of
acyclic reorientations is

\begin{lem}[{\cite[Lem.~7.1]{GZ83}}]
    Let $G = (V,E)$ be an oriented graph and $C$ a cycle basis.
    Then the connected components of $\ker C \setminus \{ p \in \R^{E} : p_e =
    0 \text{ for some } e \in E \}$ are in bijection with the acyclic
    reorientations of $G$.
\end{lem}

Tensions can be parametrized by the edges in a spanning forest, leaving it to
the non-forest edges $\Tc$ to compensate along each cycle in the cycle basis
$C$. This yields the description of an inside-out polytope analogous to the
flow case in Section~\ref{sec:FlowsInsideOut}.

\section{An interpretation of the Tutte polynomial}
\label{sec:TutteInterpretation}

An (integral) \emph{generalized Tutte--Grothendieck (T-G) invariant} is an
assignment $f_G \in \Z$ to every graph $G$ such that for some constants
$\tau_f, \sigma_f \in \Z$
\begin{enumerate}
    \item $f_G = \sigma_f f_{G_{\setminus e}} + \tau_f f_{G_{/ e}}$ for every
        $e \in E$ that is neither a loop nor a coloop, and
    \item $f_G = f_{G[e]} \cdot f_{G_{\setminus e}}$ otherwise.
\end{enumerate}

A large collection of important invariants for graphs (or more generally
matroids) qualifies as Tutte--Grothendieck invariants (cf.~\cite{BO92}), among
them all evaluations of chromatic and modular flow and tension polynomials.
The \emph{Tutte polynomial} $t_G(x,y) \in \Z[x,y]$ of a graph $G$ is the
unique function such that every generalized T-G invariant $f_G$ can be
expressed as 
\[
f_G = \sigma^{\xi(G)}_f \tau^{|V|-c(G)}_f
t_G(\tfrac{f_I}{\tau_f},\tfrac{f_L}{\sigma_f})
\]
where $f_L$ and $f_I$ is the invariant of a loop and coloop, respectively.

Insofar the Tutte polynomial expresses a multitude of enumerative invariants,
but, to the best of our knowledge, there is no enumerative interpretation for
arbitrary evaluations of the Tutte polynomial.  
The
reciprocity statements of Sections~\ref{sec:Flows} and \ref{sec:Tensions} yield
a natural interpretation of what the Tutte polynomial evaluated at $(1+\l,1+k)$
for positive $k,\l$ counts.

\begin{thm} \label{thm:TutteInterpretation}
    Let $G$ be a graph and $k,\l$ two positive integers. Then $t_G(1+\l,1+k)$
    counts triples $(f,t,\s)$ where     
    \begin{enumerate}[i)]
        \item $f$ is a $\Zk$-flow and $t$ is a $\Zl$-tension on $G$,
        \item $f$ and $t$ have disjoint support, and
        \item $\s \subseteq E\setminus\supp(f)\cup\supp(t)$ is a reorientation of $G_{\setminus \supp(f) \cup
            \supp(t)}$.
    \end{enumerate}
\end{thm} 

The polynomial  $t_G(1+\l,1+k)$ is also known as the rank polynomial of $G$.

In Appendix~\ref{sec:CombTutteProof}, we give a proof of
Theorem~\ref{thm:TutteInterpretation} from first principles, i.e.~we prove
that the stated cardinality itself is a generalized T-G invariant with
structure constants $\tau = \sigma = 1$ and values $1+k$ and $1+\l$ for loops
and coloops, respectively.  Here, however, we give a proof by noting that
Theorem~\ref{thm:TutteInterpretation} is equivalent to the convolution formula
for Tutte polynomials of Kook, Reiner, and Stanton \cite{KRS99} specialized to
graphs.

We need the following observation regarding reorientations of graphs.

\begin{lem}\label{lem:AcyclicTotCyclic}
  Let $G = (V,E)$ be an oriented graph. Then there is a unique $S \subseteq E$
  such that $G[S]$ is totally cyclic and $G_{/ S}$ is acyclic.
\end{lem}
\begin{proof}
  Let $S \subseteq E$ be the collection of edges that lie on a directed cycle
  in $G$. Then, clearly, $G[S]$ is totally cyclic and, as $G[S]$ is
  componentwise strongly connected, $G_{/ S}$ is acyclic. As for uniqueness,
  suppose that $S^\prime \subseteq E$ has the same properties. As $G[S']$ is
  totally cyclic, clearly $S'\subset S$. Now if $e \in S \setminus S^\prime$,
  then $e$ is contained in a directed cycle which remains true in $G_{/
  S^\prime}$, in contradiction to $G_{/ S^\prime}$ being acyclic. Hence
  $S=S'$.
\end{proof}

\begin{proof}[Proof of Theorem~\ref{thm:TutteInterpretation}]
    Consider the collection of triples $(f,t,\s)$ as in the theorem. For every
    triple, we claim that there is a unique $S \subseteq E$ such that 
    \begin{compactenum}
    \item $f$ is a $\Zk$-flow on $G[S]$ and $\sigma \cap S$ is a totally
        cyclic reorientation of $(G[S])_{/ \supp(f)}$, and
    \item $t$ is a $\Zl$-tension on $G_{/ S}$ such that $\sigma \setminus S$
        is an acyclic reorientation of $(G_{/ S})_{\setminus \supp(t)}$.
    \end{compactenum}
    Indeed, let $S^\prime \subseteq E$ be the set for $_\s
    G_{/\supp(f)\setminus\supp(t)}$ whose existence and uniqueness is asserted
    by Lemma~\ref{lem:AcyclicTotCyclic}. Now, it is easy to verify that $S :=
    S^\prime \cup \supp(f)$ is the unique set with the properties above.

    In light of Theorem \ref{thm:MainFlow} and Theorem
    \ref{thm:TensionReciprocity}, the number of triples for which $S$ is the
    unique set obeying the above properties is
    \[
    (-1)^{\xi(G[S])}\mFlow_{G[S]}(-k) \cdot (-1)^{|V| - c(G_{/ S})} \mTen_{G_{/
    S}}(-\l)
    \]
    The flow and tension polynomials are both specializations of the Tutte
    polynomial and hence
    \begin{align*}
        t_{G[S]}(0,1+k) &= (-1)^{\xi(G[S])}\mFlow_{G[S]}(-k)\\
        t_{G_{/ S}}(1+\l,0) &= (-1)^{|V| - c(G_{/ S})} \mTen_{G_{/S}}(-\l).
    \end{align*}
    To finish the proof, we recall the result of Kook, Reiner, and Stanton
    \cite[Thm.~1]{KRS99}\note{fix ref} stating that
    \[
        t_G(1+\l,1+k) = \sum_{S \subseteq E} t_{G[S]}(0,1+k)\,t_{G_{/S}}(1+\l,0).
    \]
\end{proof}

This interpretation yields the following counting formula which matches a
result of Reiner~\cite{reiner99} but removes the restriction to prime powers.
It follows from Theorem~\ref{thm:TutteInterpretation} together with the fact
that every flow or tension is nowhere zero restricted to its support.

\begin{cor}
    Let $G = (V,E)$ be a graph, then the Tutte polynomial is given by
    \[
    t_G(1+\l,1+k) = \sum_{S \subseteq T \subseteq E}
    2^{|T \setminus S|} \mFlow_{G[S]}(k) \mTen_{G_{/T}}(\l)
    \]
\end{cor}

Let us remark that enumerative interpretations for \emph{all} evaluations of
the Tutte polynomial are not to be expected as for negative parameters the
sign of $t_G(x,y)$ depends on the magnitude of the arguments. However, our
interpretation misses some fundamental evaluations such as $t_G(1,2)$ for the
number of spanning sets and $t_G(2,1)$ for the number of independent sets.  We
also remark that interpretations for evaluations of the flow- and tension
polynomials at negative values also yield \emph{interpretations} in the spirit
of \cite[Cor.~2]{reiner99} for the missing two orthants $(1-\l,1+k)$ and
$(1+\l,1-k)$.

\appendix

\section{Combinatorial Proof of Modular Flow Reciprocity}
\label{sec:comb-proof-for-flows}

In this section we give a combinatorial proof of Theorem~\ref{thm:MainFlow}.
Our approach is straightforward: we show that $(-1)^{\xi(G)}\mFlow_G(-k)$ is a
generalized Tutte--Grothendieck invariant with the correct structure
constants.

Let $\rmFlowSet_G(k)$ denote the set of all pairs $(f,\sigma)$ of a $\Zk$-flow
$f$ and a totally cyclic reorientation $\sigma$ of $G_{/\supp(f)}$. Using this
notation Theorem~\ref{thm:MainFlow} simply states
\begin{eqnarray}
\label{eqn:FlowReciprocity}
(-1)^{\xi(G)}\mFlow_G(-k)=\rmFlow_G(k).
\end{eqnarray}
In light of Proposition~\ref{prop:FlowRecursion} it suffices to show that
$\rmFlow_G(k)$ is a Tutte--Grothendieck invariant with the structure constants as
given in the following theorem.

\begin{thm}
\label{thm:rmFlowRecursion}
Let $G=(V,E)$ be an oriented graph and let $k\in\N$.
\begin{enumerate}
\item If $E=\emptyset$, then $\rmFlow_G(k)=1$.
\item If $e\in E$ is a coloop, then $\rmFlow_G(k)=0$.
\item If $e\in E$ is a loop, then $\rmFlow_G(k)=(k+1)\cdot\rmFlow_{G\setminus
    e}(k)$.
\item If $e\in E$ is neither a loop nor a coloop, then $\rmFlow_G(k)=\rmFlow_{G\setminus
    e}(k)+\rmFlow_{G/e}(k)$.
\end{enumerate}
\end{thm}

To show this theorem we examine how a $\Zk$-flow on $G$ induces
$\Zk$-flows on $G_{/e}$ and $G_{\setminus e}$, respectively, and how a
totally cyclic reorientation of $G$ induces totally cyclic reorientations of
$G_{/e}$ and $G_{\setminus e}$, respectively. We first turn our attention to the
$\Zk$-flows.

\begin{lem} 
\label{lem:flow-projection}
Let $G=(V,E)$ be an oriented graph and $e\in E$ neither a loop nor a coloop. If
$f$ is a $\Zk$-flow on $G$, then $f|_{E\setminus e}$ is
\begin{enumerate}
\item a $\Zk$-flow on $G_{/e}$ and
\item a $\Zk$-flow on $G_{\setminus e}$ if and only if $f(e)=0$.
\end{enumerate}
\end{lem}

\begin{proof}
  Let $e=uv$. At any vertex $w\not\in\{u,v\}$ the flow $(Af)_v$ does not change
  when passing from $G$ to $G_{/e}$ or $G_{\setminus e}$. In $G_{/e}$ the
  vertices $u$ and $v$ have been identified to form a vertex $u'$ and
  $(A_{G_{/e}}f|_{E\setminus e})_{u'}= (Af)_u + (Af)_v = 0$. In $G_{\setminus
    e}$ we have $(A_{G_{\setminus e}}f|_{E\setminus e})_{u'}= (Af)_u - f(e)$
  which is zero if and only if $f(e)=0$, and similarly for $v$.
\end{proof}

So a $\Zk$-flow on $G$ induces a $\Zk$-flow on $G_{/e}$ and if $f(e)=0$ it
also induces a $\Zk$-flow on $G_{\setminus e}$. Moreover it turns out that any
$\Zk$-flow on $G_{/e}$ is induced by a unique $\Zk$-flow on $G$ and the same
holds for $G_{\setminus e}$.

\begin{lem}
\label{lem:flow-lifting}
Let $G=(V,E)$ be an oriented graph and $e\in E$ neither a loop nor a coloop.
\begin{enumerate}
\item Given a $\Zk$-flow $f'$ on $G_{/e}$ there is a unique $\Zk$-flow $f$ on
  $G$ such that $f|_{E\setminus e}=f'$.
\item Given a $\Zk$-flow $f'$ on $G_{\setminus e}$ there is a unique $\Zk$-flow
  $f$ on $G$ such that $f|_{E\setminus e}=f'$. Moreover this flow has the
  property $f(e)=0$.
\end{enumerate}
\end{lem}

\begin{proof} 
  In both cases, we necessarily have $f(e')=f'(e')$ for all $e'\not=e$ and we
  have to check that there is unique choice for $f(e)$ that makes $f$ a
  $\Zk$-flow. Let $e=uv$ oriented from $u$ to $v$. Let $A^*$ denote the
  incidence matrix of $G$ with the column corresponding to $e$ removed. In both
  cases $(Af)_u = (A^*f)_u - f(e)$ and $(Af)_v = (A^*f)_v + f(e)$. So $f$ is a
  $\Zk$-flow if and only if $f(e)=(A^*f)_u$ and $f(e)=-(A^*f)_v$. In the first
  case these two values coincide because $(A^*f)_u+(A^*f)_v =
  (A_{G_{/e}}f')_{u'} =0$ where $u'$ is the vertex obtained by identifying $u$
  and $v$. In the second case, both of these values are zero, because $(A^*f)_u
  = (A_{G_{\setminus e}}f')_u=0$ and $(A^*f)_v = (A_{G_{\setminus e}}f')_v=0$.
\end{proof}

Now we turn to totally cyclic reorientations. Here the situation is a bit more
complicated compared to $\Zk$-flows. We start with a useful characterization of
totally cyclic orientations.

\begin{lem}
\label{lem:existence-of-dipath}
Let $G$ be an oriented graph. $\sigma$ is a totally cyclic reorientation of $G$
if and only if for any vertices $u,v\in V$ in the same component of the
underlying undirected graph, there exists a directed path in $_\sigma G$ from
$u$ to $v$.
\end{lem}

\begin{proof}
  Suppose $\sigma$ is totally cyclic. As both $u$ and $v$ lie in the same
  component of the undirected graph, there is an undirected path $P$ from $u$ to
  $v$. As $_\sigma G$ is totally cyclic, every edge of $P$ lies on a directed
  cycle. In a directed cycle, there is a directed path from any vertex to any
  other vertex. So for any edge $u_iv_i$ in $P$ there is a directed path in $G$
  from $u_i$ to $v_i$. Concatenating all these paths, we obtain a directed walk
  in $G$ from $u$ to $v$, which in particular contains a directed path from $u$
  to $v$ as a subgraph.

  Conversely, suppose we can always find a directed path from any vertex to any
  other. Let $e$ be an edge oriented from $u$ to $v$. Then the assumption
  guarantees the existence of a path $P$ from $v$ to $u$. Concatenating $P$ and
  $e$ yields a directed cycle.
\end{proof}

In the following we use $\Delta$ to denote the symmetric difference of sets. So
given a reorientation $\sigma$ and an edge $e$, $\sigma\Delta e$ is the
reorientation obtained from $\sigma$ by reversing the edge $e$.

\begin{lem}
\label{lem:tco-projection}
Let $G=(V,E)$ be an oriented graph and $e\in E$ neither a loop nor a coloop. Let
$\sigma$ be a totally cyclic reorientation of $G$. Then
\begin{enumerate}
\item $\sigma\cap (E\setminus e)$ is a totally cyclic reorientation of $G_{/e}$,
  and
\item $\sigma\cap (E\setminus e)$ is a totally cyclic reorientation of
  $G_{\setminus e}$ if and only if both $\sigma$ and $\sigma \Delta e$ are
  totally cyclic reorientations of $G$.
\end{enumerate}
\end{lem}

\begin{proof}
  \emph{(1)} As $_\sigma G$ is totally cyclic, there is a collection
  $\mathcal{C}$ of directed cycles in $_\sigma G$ that cover all edges. Then $\{
  C_{/e} | C\in\mathcal{C}\}$ is a collection of directed cycles in
  $_{\sigma\cap (E\setminus e)}G_{/e}$ that covers all edges in $G_{/e}$ and
  hence $_{\sigma\cap (E\setminus e)}G_{/e}$ is totally cyclic.

  \emph{(2)} Let $e=uv$. Suppose $_{\sigma\cap (E\setminus e)}G_{\setminus e}$
  is totally cyclic. Then by Lemma~\ref{lem:existence-of-dipath} there exist
  directed paths from $u$ to $v$ and from $v$ to $u$. These show that no matter
  which way we orient $e$, we can always find a directed cycle on which $e$ lies
  and so both $_\sigma G$ and $_{\sigma\Delta e} G$ are totally cyclic.

  Conversely, suppose both $_\sigma G$ and $_{\sigma\Delta e} G$ are totally
  cyclic. The edge $e$ lies on a directed cycle in $_{\sigma\Delta e} G$ so by
  Lemma~\ref{lem:existence-of-dipath} there is a directed path $P$ from $u$ to
  $v$ in $_\sigma G_{\setminus e}$.  Let $u',v'$ be any two vertices in
  $G_{\setminus e}$. As $_\sigma G$ is totally cyclic there is a directed path
  $P'$ in $_\sigma G$ from $u'$ to $v'$. We replace every occurrence of $e$ in
  $P'$ with $P$ and obtain a directed walk (and hence a directed path) in
  $_{\sigma\cap E\setminus e} G_{\setminus e}$ from $u'$ to $v'$. By
  Lemma~\ref{lem:existence-of-dipath} it follows that $_{\sigma\cap E\setminus
    e} G_{\setminus e}$ is totally cyclic.
\end{proof}

\begin{lem}
\label{lem:tco-lifting}
Let $G=(V,E)$ be an oriented graph and $e\in E$ neither a loop nor a coloop.
\begin{enumerate}
\item Let $\sigma\subseteq E\setminus e$ be a totally cyclic reorientation of
  $G_{/e}$. Then at least one of ${\sigma}$ and ${\sigma\cup e}$ is a totally
  cyclic reorientation of $G$.
\item Let $\sigma\subseteq E\setminus e$ be a totally cyclic reorientation of
  $G_{\setminus e}$. Then both ${\sigma}$ and $\sigma\cup e$ are totally cyclic
  reorientations of $G$.
\end{enumerate}
\end{lem}

\begin{proof}
  \emph{(1)} Let $_\sigma G_{/e}$ be totally cyclic and $e=uv$. Let
  $\mathcal{C}$ be a collection of directed cycles in $_\sigma G_{/e}$ that
  covers all edges of $G_{/e}$. Now we distinguish two cases: Is one of these
  cycles ``broken'' in $G$ or not? More precisely does there exist a cycle
  $C\in\mathcal{C}$ that contains consecutive edges $e_1$ and $e_2$ such that
  $e_1$ enters $u$ and $e_2$ leaves $v$ (or vice versa)?\footnote{We also
    require that $C$ does not consist of a single edge that is a loop.} If not,
  then $\mathcal{C}$ shows that $_\sigma G_{\setminus e}$ is also totally cyclic
  and we can continue as in part \emph{(2)} below.

  So we suppose that $C$ is such a broken cycle. In this case $C$ gives a
  directed path from $v$ to $u$ in $G$. We now orient $e$ from $u$ to $v$. Then
  any directed path $P$ in $G_{/e}$ from a vertex $u'$ to a vertex $v'$ can be
  turned into a directed path in $G$ from $u'$ to $v'$ by substituting the edge
  $e$ or the path given by $C$ wherever $P$ is broken. Using
  Lemma~\ref{lem:existence-of-dipath} the claim follows.

  \emph{(2)} Already in $G_{\setminus e}$ there is, for any two vertices $u,v$
  in the same component, a directed path from $u$ to $v$. This remains true
  after the edge $e$ is inserted, no matter how $e$ is oriented (note that $e$
  is not a coloop). So by Lemma~\ref{lem:existence-of-dipath} both ${\sigma}$
  and $\sigma\cup e$ are totally cyclic reorientations of $G$.
\end{proof}

Now we have all ingredients to show that $\rmFlow_G(k)$ is a
Tutte--Grothendieck invariant.

\begin{proof}[Proof of Theorem~\ref{thm:rmFlowRecursion}]
  \emph{1.}  If $E=\emptyset$, then $\rmFlowSet_G(k)=\{(\emptyset,\emptyset)\}$.

  \emph{2.}  If $e\in E$ is a coloop, then any flow $f$ on $G$ has
  $f(e)=0$. Thus $e$ is also a coloop in $G/\supp(f)$ which means that there
  is no totally cyclic orientation on $G/\supp(f)$. So
  $\rmFlowSet_G(k)=\emptyset$.

  \emph{3.} If $e\in E$ is a loop, then $(f,\sigma) \mapsto (f|_{E\setminus e},
  \sigma\cap E\setminus e)$ is a surjective map from $\rmFlowSet_G(k)$ onto
  $\rmFlowSet_{G\setminus e}$ and every fiber of this map has cardinality
  $k+1$. The reason is that given $(f|_{E\setminus e}, \sigma\cap E\setminus e)$
  we can define $f(e)\in\Zk$ arbitrarily and $f$ will become a $\Zk$-flow on
  $G$. The case $f(e)=0$ is counted twice as either orientation of $e$ will turn
  $\sigma$ into a totally cyclic orientation of $G_{/\supp(f)}$.

  \emph{4.} Let $e\in E$ be neither a coloop nor a loop. Consider the map
  $\pi_{G/e}:\rmFlowSet_G(k)\rar \rmFlowSet_{G/e}(k)$ given by
  $(f,\sigma)\mapsto (f|_{E\setminus e},\sigma\cap E\setminus
  e)$. Lemmas~\ref{lem:flow-projection} and \ref{lem:tco-projection} tell us
  that $\pi_{G/e}$ is well-defined and Lemmas~\ref{lem:flow-lifting} and
  \ref{lem:tco-projection} tell us that every $(f',\sigma')\in
  \rmFlowSet_{G/e}(k)$ has either one or two pre-images under
  $\pi_{G/e}$. $(f',\sigma')$ has two pre-images if and only if the unique
  $\Zk$-flow $f$ with $f|_{E\setminus e}=f'$ has $f(e)=0$ and both $\sigma'$ and
  $\sigma'\cup e$ are totally cyclic reorientations of $G_{/supp(f)}$.

  Loosely speaking, this means that the cardinalities of $\rmFlowSet_G(k)$ and
  $\rmFlowSet_{G/e}(k)$ are the same, except that we have to count those
  $(f',\sigma')\in \rmFlowSet_{G/e}(k)$ that have two pre-images twice.

  So let $\rmFlowSet'_{G}(k)$ denote the set of all $(f,\sigma)\in \rmFlowSet_G(k)$ such
  that $f(e)=0$ and both $\sigma$ and $\sigma\Delta e$ are totally
  cyclic reorientations on $G_{/\supp f}$. Consider the map
  $\pi_{G\setminus e}:\rmFlowSet'_G(k) \rar \rmFlowSet_{G\setminus e}(k)$ given by
  $(f,\sigma)\mapsto (f|_{E\setminus e},\sigma\cap E\setminus
  e)$. Lemmas~\ref{lem:flow-projection} and \ref{lem:tco-projection}
  tell us that $\pi_{G \setminus e}$ is well-defined and
  Lemmas~\ref{lem:flow-lifting} and \ref{lem:tco-projection} tell us
  that every $(f',\sigma')\in F_{G\setminus e}(k)$ has precisely two
  pre-images under $\pi_{G\setminus e}$. But this means that
  $\rmFlow_G(k)=\rmFlow_{G/e}(k) + \rmFlow_{G\setminus e}(k)$ as desired.
\end{proof}

\section{Combinatorial Proof of the Tutte Interpretation}
\label{sec:CombTutteProof}

In this section we give a combinatorial proof of
Theorem~\ref{thm:TutteInterpretation}, our interpretation of
$t_G(1+\l,1+k)$. The approach is similar to that in
Appendix~\ref{sec:comb-proof-for-flows}: we show that the counting function,
that we claim is identical to the Tutte polynomial, is a Tutte--Grothendieck
invariant with the appropriate structure constants. Surprisingly, the
combinatorial proof of Theorem~\ref{thm:TutteInterpretation} is much simpler
than the combinatorial proof of Theorem~\ref{thm:MainFlow}.

Theorem~\ref{thm:TutteInterpretation} states that $t_G(1+\l,1+k)$ counts the
number of triples $(f,t,\sigma)$ where $f$ and $t$ are, respectively, a
$\Z_k$-flow and a $\Z_l$-tension on $G$ with disjoint support and
$\sigma\subseteq E\setminus\supp(f)\cup\supp(t)$. Now, for any edge set
$S\subseteq E$, the $\Z_k$-flows $f$ on $G$ with $f(e)=0$ for all $e\in S$ are
in bijection with the $\Z_k$-flows on $G_{\setminus S}$. Correspondingly, for
any edge set $S\subseteq E$, the $\Z_l$-tensions $t$ on $G$ with $t(e)=0$ for all
$e\in S$ are in bijection with the $\Z_l$-tensions on $G_{/ S}$. So if we define
the sets $\tcs_G(\l,k)$ by
\begin{eqnarray*}
\tcs_G(\l,k) = \{ (S,t,f) &:& 
S\subseteq E, \\ &&
\text{$t$ a $\Z_\l$-tension on $G/S$}, \\ &&
\text{$f$ a $\Z_k$-flow on $G[S]$}\},
\end{eqnarray*}
for $k,\l\in\N$, Theorem~\ref{thm:TutteInterpretation} then becomes
$t_G(1+\l,1+k)=\tcf_G(\l,k)$ for all $\l,k\geq 1$. By the fact that the Tutte
polynomial is a Tutte--Grothendieck invariant, all we have to show is the
following:
\begin{enumerate}
\item If $E=\emptyset$, then $\tcf_G(\l,k)=1$.
\item If $e\in E$ is a coloop, then $\tcf_G(\l,k)=(1+\l)\cdot \tcf_{G/e}(\l,k)$.
\item If $e\in E$ is a loop, then $\tcf_G(\l,k)=(1+k)\cdot \tcf_{G\setminus
    e}(\l,k)$.
\item If $e\in E$ is neither a loop nor a coloop, then
  $\tcf_G(\l,k)=\tcf_{G/e}(\l,k)+ \tcf_{G\setminus e}(\l,k)$.
\end{enumerate}
For any statement $A$ we will denote by $[A]$ the number $1$ if $A$ holds and
$0$ if $A$ does not hold. Using this shorthand notation and the fact that if $e$ is a
loop or a coloop then $t_{G_{\setminus e}}=t_{G_{/e}}$, we can write what we
have to show more compactly as
\begin{eqnarray}
\label{eqn:TutteRecursion}
\tcf_G(\l,k)=\l^{[\text{$e$ is a coloop}]}\tcf_{G\setminus e}(\l,k)+k^{[\text{$e$ is a loop}]}\tcf_{G/ e}(\l,k).
\end{eqnarray}

Before we show that this identity holds, we work out how the
$\Z_l$-tensions on $G$ and on $G_{\setminus e}$ are related, just as we did in
Appendix~\ref{sec:comb-proof-for-flows} for $\Zk$-flows.

Given a map $f:E\rar\Zk$ and a set $S\subseteq E$ we define $f|_{G_{/S}}$ and
$f|_{G_{\setminus S}}$ to be the maps obtained by restricting $f$ to the
respective edge sets of $G_{/S}$ and $G_{\setminus S}$. A fiber of a map $f$ is
the set $f^{-1}(x)$ for any $x$ in the image.

\begin{lem}
\label{lem:fiber-tension}
If $t$ is a $\Z_\l$-tension on $G$, then $t|_{G\setminus e}$ is a
$\Z_\l$-tension on $G\setminus e$. If $e$ is a coloop, every fiber of the map
$t\mapsto t|_{G\setminus e}$ has cardinality $k$. Otherwise every fiber of the
map $t\mapsto t|_{G\setminus e}$ has cardinality $1$.
\end{lem}

\begin{proof}

  A cycle in $G_{\setminus e}$ is also a cycle in $G$. If $\sprod{C}{t}=0$ holds
  for every cycle $C$ of $G$, then it also holds for every cycle of
  $G_{\setminus e}$. So $t|_{G_{\setminus e}}$ is a $\Zk$-tension on
  $G_{\setminus e}$.

Now suppose $e$ is not a coloop in $G$. Let $t'$ be a $\Zk$-tension on
$G_{\setminus e}$. Which $\Z_\l$-tensions $t$ on $G$ have $t|_{G_{\setminus
    e}}=t'$? Necessarily, $t(e'):=t'(e')$ for all $e'\not=e$. All we have to
show is that there is a unique choice of $t(e)$ such that $t$ is a tension. Now
as $e$ is not a coloop, $e$ lies on a cycle $C$. The weights of all other edges
on $C$ are fixed. As $\Zk$ is a group, there is a unique choice of $t(e)$ such
that $\sprod{C}{t}=0$. $t(e)$ does not depend on the choice of $C$, as
$\sprod{C'}{t}=\sprod{C'}{t'}=0$ for all cycles $C'$ that do not contain $e$.

If $e$ is a coloop, then $e$ does not lie on any cycle and so we can choose
$t(e)\in\Zk$ arbitrarily.
\end{proof}

Now the proof of our interpretation of the Tutte polynomial is easy.

\begin{proof}[Proof of Theorem~\ref{thm:TutteInterpretation}]
  We have to show that (\ref{eqn:TutteRecursion}) holds. To that end we define a
  map
\begin{eqnarray*}
\tcs_G(\l,k) & \rar & \tcs_{G\setminus e}(\l,k)\uplus\tcs_{G/ e}(\l,k)\\
(S,t,f) & \mapsto & \left\{
\begin{array}{lcr}
(S,t|_{G\setminus e},f) & \text{if} & e\not\in S, \\
(S\setminus e,t,f|_{G/e}) & \text{if} & e\in S.
\end{array}\right.
\end{eqnarray*}
By Lemma~\ref{lem:fiber-tension} a fiber over $\tcs_{G\setminus e}(\l,k)$ has
cardinality $\l$ if $e$ is a coloop and cardinality 1 otherwise. As we have seen
in Appendix~\ref{sec:comb-proof-for-flows}, a fiber over $\tcs_{G/ e}(\l,k)$ has
cardinality $k$ if $e$ is a loop and cardinality 1 otherwise. Thus
\[
   \tcf_G(\l,k)=\l^{[\text{$e$ is a coloop}]}\tcf_{G\setminus e}(\l,k)+k^{[\text{$e$ is a loop}]}\tcf_{G/ e}(\l,k)
\]
for any $e\in E$. 
\end{proof}

\bibliographystyle{siam}
\bibliography{ModularFlows}

\end{document}